\documentclass{article}
\usepackage{graphicx}
\usepackage{epstopdf}
\usepackage{amsmath}
\usepackage{amssymb}
\usepackage{bm}
\usepackage{bbm}
\usepackage{enumerate}
\usepackage[T1]{fontenc}
\usepackage{hyperref}
\usepackage{color}
\usepackage{curves}
\usepackage{tikz}
\usepackage[center]{caption}
\usepackage{ntheorem}
\usepackage{float}
\usepackage{mathrsfs}
\usepackage{wrapfig}

\usepackage{todonotes}
\usepackage[normalem]{ulem}

\numberwithin{figure}{section}
\numberwithin{equation}{section}

\title{Strictly weak consensus in the\\ uniform compass model on $\Z$}
\author{Nina Gantert
	\\\normalsize Technische Universit\"at M\"unchen\\
	Markus Heydenreich
	\\\normalsize Ludwig-Maximilians-Universit\"at M\"unchen\\
	Timo Hirscher\thanks{Research supported by grants from the Swedish Research Council and the
	Royal Swedish Academy of Sciences.}
\\\normalsize Stockholms Universitet}

\theoremstyle{break}
\newtheorem{theorem}{Theorem}[section]
\newtheorem{corollary}{Corollary}[section]
\newtheorem{lemma}{Lemma}[section]
\newtheorem{proposition}{Proposition}[section]
\newtheorem{definition}{Definition}
\newtheorem{remark}{Remark}[section]
\newtheorem{example}{Example}[section]
\makeatletter
\let\c@proposition\c@theorem
\let\c@lemma\c@theorem
\let\c@example\c@theorem
\let\c@corollary\c@theorem
\let\c@remark\c@theorem
\makeatother

\newenvironment{proof}{\noindent{\emph{Proof. }}}{\vspace{-0.5cm}~\hfill $\square$\vspace{0.5cm}}
\newenvironment{nproof}[1]{\noindent{\emph{Proof #1.}}}{\vspace{-1em}~\hfill $\square$\vspace{2em}}

\newcommand\N{\mathbb{N}}
\newcommand\R{\mathbb{R}}
\newcommand\Z{\mathbb{Z}}
\renewcommand\S{\mathcal{S}}
\newcommand\E{\mathbb{E}\,}
\newcommand\Prob{\mathbb{P}}
\renewcommand\epsilon{\varepsilon}
\renewcommand\phi{\varphi}

\definecolor{darkblue}{rgb}{0,0,.5}
\definecolor{darkgreen}{rgb}{0,.6,.3}
\hypersetup{colorlinks=true, breaklinks=true, linkcolor=blue, 
citecolor=darkblue, menucolor=blue, urlcolor=blue}

\begin{document}
\newpage
\maketitle
\begin{abstract}
We investigate a model for opinion dynamics, where individuals (modeled by vertices of a graph) hold certain abstract opinions. As time progresses, neighboring individuals interact with each other, and this interaction results in a realignment of opinions closer towards each other. This mechanism triggers formation of consensus among the individuals. Our main focus is on \emph{strong consensus} (i.e.\ global agreement of all individuals) versus \emph{weak consensus} (i.e.\ local agreement among neighbors). By extending a known model to a more
general opinion space, which lacks a ``central'' opinion acting as a contraction point, we provide an example of an opinion formation process on the 
one-dimensional lattice 
with weak consensus but no strong consensus.
\end{abstract}



\section{Introduction}\label{intro}
\paragraph{Background.}
	A major theme of statistical physics is to derive macroscopic properties of a system from simple interactions at the microscopic level.
	A prime example is the well-known Ising model, where the strength of mutual influence of neighboring magnetic dipoles depends on
	the temperature. While at high temperature the state is incoherent and chaotic so that the mean magnetization is 0,
	at low temperature, the spins align collectively and form a macroscopic magnet.
	
	Transitions from individual to collective behavior, as observed in interacting particle systems like the Ising model, attracted the attention
	of social sciences. Despite being overly simplistic, an abundance of similar but qualitatively different interacting particle systems (such
	as the voter model or the majority rule model or the contact process) were introduced in order to describe and explain group behavior
	and swarm phenomena, which can be observed in real life. A broad overview of models, which fall into the research field commonly
	known as {\itshape opinion dynamics}, can be found in the survey article ``{\em Statistical physics of social dynamics}'' by Castellano
	et al.\ \cite{surv}. For more mathematical background of models in this spirit we refer to Liggett's monograph \cite{Liggett}.

The model which Deffuant et al.\ \cite{Model} introduced almost 20 years ago is a simple representative
of the so-called {\em bounded confidence models}: An opinion is represented by a real number and neighboring
agents update their opinions in pairwise interactions towards a compromise, but only if the opinions
with which they enter the interaction do not differ by more than a given threshold. This is supposed to
shape the phenomenon that humans in general are inclined to modify their opinion on a specific topic
when confronted with arguments differing from their own belief, but openness of mind is lost if a priori
the opinions are differing too much. 

A rigorous and comprehensive mathematical understanding of bounded confidence models such as
the one introduced by Deffuant et al.\ on infinite graphs (in particular on grids of dimension greater than 1)
is still lacking. 

\paragraph{A mathematical model for opinion dynamics.} 

We now describe the Deffuant model as a continuous-time Markov process. To this end, we consider a connected and locally finite graph $G=(V,E)$, where the vertices are interpreted as \emph{individuals} or \emph{agents}. Our graphs will always be undirected and two individuals interact whenever they are linked by an edge.
We further denote by $\mathcal{S}$ a compact and convex space of \emph{opinions} with metric $d$, and the state space of the Markov process is given by $\Omega=\mathcal S^{V}$ (equipped with the product topology). 
For given parameters  $\mu\in(0,\tfrac 12]$ and $\theta>0$, the dynamics of the process is described by the probability generator 
\begin{equation}\label{eqDefLf}
	\mathcal Lf(\eta)=\sum_{e\in E}\left(f(A_e\eta)-f(\eta)\right), \qquad \eta\in\Omega,
\end{equation}
where $f$ is a continuous test function, and the operator $A_e$ for the edge $e=\langle u,v\rangle$ acts on $\eta\in\Omega$ as 
\begin{equation}\label{eqDefAe}
	A_e\eta(v)=
	\begin{cases} 
		\eta(v)&\text{if }v\not\in e;\\ 
		\eta(v)+\mu\,\mathbbm{1}_{\{d(\eta(v),\eta(u))\le \theta\}}\big(\eta(u)-\eta(v)\big)
		&\text{if } e=\langle u,v\rangle.
	\end{cases}
\end{equation}
Mind that, given $\eta\in\Omega$, the convexity of $\mathcal{S}$ implies $A_e\eta\in\Omega$ for all $e\in E$. 
Existence and uniqueness of a Feller process having $\mathcal L$ as its generator is standard, cf.\ Chapter IX in \cite{Liggett}.   

The dynamics defined in \eqref{eqDefLf} and \eqref{eqDefAe} can best be explained via the graphical construction: On every edge $e$ there is an independent Poisson clock. Upon clock rings, the two incident individuals interact, and the result of the interaction is as follows: if the opinions differ by at most $\theta$, then the two individuals alter their opinions and move both a proportion $\mu$ closer towards each other (in the extreme case $\mu=1/2$, they even agree on the average opinion). If, however, the opinions differ by more than the ``confidence bound'' $\theta$, then there is no change. 

One of the main questions related to this model is the following: Given an initial distribution, will the opinions of different individuals align as $t\to\infty$ (we call this \emph{consensus}) or not? Our prime example for $G=(V,E)$ is the \emph{two-sided infinite path} with $V=\Z$ and $E=\{\langle v,v+1\rangle, v\in\Z\}$. 

\paragraph{Previous work.}
In 2011, Lanchier \cite{Lanchier} was the first to publish a result about the standard Deffuant model on $\Z$:
For i.i.d.\ initial opinions that are uniform on $[0,1]$, he proved that there is
a sharp phase transition at $\theta=\frac12$, from almost sure no consensus in
the subcritical regime ($\theta<1/2$) to almost sure consensus in the supercritical regime ($\theta>1/2$), irrespectively of the value of $\mu$.
In the same year, using different techniques, H\"aggstr\"om \cite{ShareDrink} reproved and slightly sharpened Lanchier's result: He showed in
addition to it that in the supercritical regime, the almost sure consensus is not only local (i.e.\ between
neighbors, cf.\ {\em weak} in Definition \ref{states}) but global (corresponding to {\em strong} in Definition
\ref{states}) with $\frac12$ as deterministic limit for each individual opinion.
Later these results were extended beyond the uniform distribution on $[0,1]$ for the initial opinions, first to general univariate distributions by H\"aggstr\"om and Hirscher \cite{Deffuant}, then to vector-valued \cite{multidim} and measure-valued opinions \cite{measure} by Hirscher.

In addition to the Euclidean norm, other measures of distance of two opinions were proposed and
analyzed, however, the underlying opinion space $\mathcal{S}$ considered was always convex: $\R^n$ for some
$n\in\N$ in the finite-dimensional case and the set of probability densities on $[0,1]$
in the measure-valued case.

\paragraph{The compass model.}
	In contrast to the standard Deffuant model and its generalizations described above, we want to consider opinion spaces
	that are not necessarily convex but only path-connected, and shall see that this modification can change the limiting
    behavior fundamentally; further, for simplicity, we set $\theta=\infty$, in other words ignore the ``confidence bound'' modeled by the
    parameter $\theta$). This modification is motivated by extensions of the Ising model to more general state spaces, e.g.\ the unit circle,
    see the book \cite{FriedVelen17}. We will come back to related models and give more details towards the end of the introduction.

    For a non-convex opinion space $\mathcal{S}$, the interaction rule laid down in	\eqref{eqDefAe} has to be adapted so that
    updates do not lead out of $\mathcal{S}$. The arguably most natural way to achieve
   	this is to measure distance between two opinions as the length of their geodesic (with respect to a metric $d$ on $\mathcal S$), along which
   	compromising agents then align their opinions, cf.\ Figure \ref{notconvex}. In case the geodesic is not unique, the selection is randomized.
		\begin{figure}[tbh]
			\centering
			\includegraphics[scale=1.2]{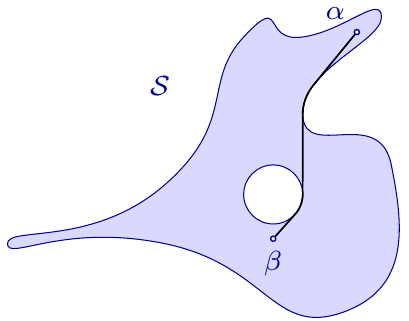}
			\caption{For non-convex $\S$, the opinions of interacting agents move towards
				each other along the geodesic between them.\label{notconvex}}
		\end{figure}
	
    Our choice for $\S$ will be the unit circle $S^1$, so that updates do not make opinions approach
	the center, but happen along their geodesic in $\S$, i.e.\ the circle arc. This change turns out to be crucial and
	the essential difference to opinion spaces considered earlier is the following:
	Given Euclidean geometry, for convex $\mathcal{S}\subseteq \R^n$ there always exists a reference point $s\in \S$
	such that 
	\begin{enumerate}[(i)]
		\item the sum of distances to $s$ of two interacting opinions can not increase through the update and
		\item $\E[d(\eta_v(t),s)]$ decreases {\em strictly} with $t$ (provided $0<\E[d(\eta_v(0),s)]<\infty$).
	\end{enumerate}

	Note at this point that convexity is a sufficient, not a necessary condition for symmetric approaches along geodesics to be contracting
    in the above sense: Also a star and the so-called {\em Lituus spiral}, given by $r(\phi)=\frac{1}{\phi^2}$ (see Figure \ref{notconvex2}),
    with respect to their centers and Euclidean geometry have this property (where in the case of the spiral, the center is not even part
    of the opinion space).

	\begin{figure}[bht]
		\centering
		\includegraphics[scale=0.9]{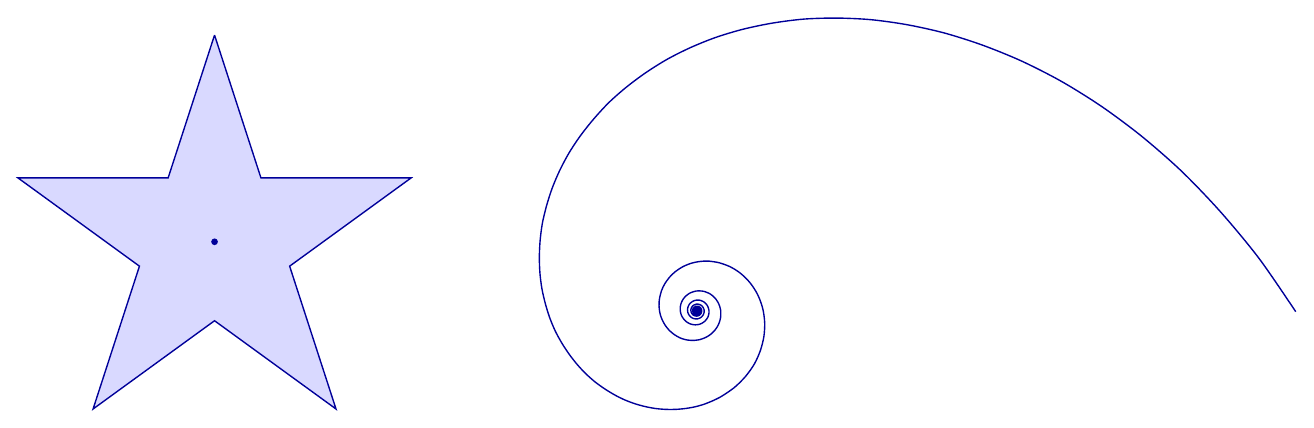}
		\caption{Both on a star and a spiral with suitable curvature, updates along geodesics are contracting towards the center,
			e.g.\ for uniform initial marginals.\label{notconvex2}}
	\end{figure}
	On $\mathcal{S}=S^1$, however, the dynamics does not have this two-part contraction property: while the (almost sure)
	weak contraction condition (i) fails for all points but the origin (i.e.\ the circle center), the strict one (ii) fails for the center with
	respect to any initial distribution.
    We will parametrize $\mathcal{S}=S^1$ via the quotient space
	$\left.\raisebox{.2em}{$\R$}\middle/ \raisebox{-.2em}{$2\Z$}\right.$, i.e.
	$$\mathcal{S}=\big\{[x];\;-1<x\leq 1\},\quad\text{where } [x]=\{y\in\R;\;\tfrac{y-x}{2}\in\Z\},$$
	and define on it the canonical metric $d([x],[y])=\min\big\{|a-b|;\; a\in[x],\ b\in[y]\big\}$.
Since elements in $\mathcal{S}$ have an interpretation as \emph{direction} (cf.\ Figure \ref{compasspic}), we propose to call our model
the \emph{compass model}. 

For ease of notation, we simply write $\mathcal{S}=(-1,1]$ instead of using the more accurate representation by equivalence classes and
write $x\ (\mathrm{mod}\ \mathcal{S})$ to refer to the unique representative of $[x]$ in $(-1,1]$.
Note that $d$ is indeed a metric and coincides with the length of the Euclidean shortest path, if distances are taken along the circle arc
(rescaled such that the total perimeter is 2).
More precisely,
$$d\colon\mathcal{S}\times \mathcal{S}\to[0,1], \qquad
                            (x,y) \mapsto \min\{|x-y|,2-|x-y|\}.$$

As indicated above, we change the dynamics to happen along geodesics in $\S$. To this end, we consider the Markov process with
a generator similar as in \eqref{eqDefLf}, namely  
\begin{equation}\label{eqDefLf2}
	\mathcal Lf(\eta)=\sum_{e\in E}\,\Big(\tfrac12\big[f(A^{(1)}_e\eta)+f(A^{(2)}_e\eta)\big]-f(\eta)\Big), \qquad \eta\in\Omega,
\end{equation}
with 
\begin{equation}\label{eqDefAe2}
	A^{(k)}_e\eta(v)=
	\begin{cases} 
		\eta(v)
		&\hskip-10em\text{if }v\not\in e;\\
		\eta(v)+\mu\big(\eta(u)-\eta(v)\big)
		&\hskip-10em\text{if }e=\langle u,v\rangle,\ |\eta(u)-\eta(v)|<1;\\
		\eta(v)+\mu\big(2-|\eta(u)-\eta(v)|\big)\, \mathrm{sgn}(\eta(v))
		\pmod{\mathcal{S}}\\
		&\hskip-10em\text{if } e=\langle u,v\rangle,\ |\eta(u)-\eta(v)|>1;\\
		\eta(v)+(-1)^{k}\mu\, \mathrm{sgn}(\eta(v))\pmod{\mathcal{S}}\\
		&\hskip-10em\text{if } e=\langle u,v\rangle,\ |\eta(u)-\eta(v)|=1,\\
	\end{cases}
\end{equation}
for $k\in\{1,2\}$, where $\mathrm{sgn}(x)=\mathbbm{1}_{\{x>0\}}-\mathbbm{1}_{\{x<0\}}$ is the sign function.\vspace{1em}

In contrast to \eqref{eqDefLf}, the jump part in \eqref{eqDefLf2} is split up into two contributions, $A_e^{(1)}$ and $A_e^{(2)}$.
This is necessary to implement a (uniformly) random choice of geodesic in the case when $|\eta(u)-\eta(v)|=1$ (i.e.\ when the two
interacting opinions are diametrically opposed). In this way, a rotational symmetry on the opinion space is preserved by the dynamics
as it does not depend on the parametrization of $\mathcal{S}$. Note, however, that for absolutely continuous initial
distributions, diametrically opposed opinions will a.s.\ not occur.

Informally, there are independent Poisson clocks on all edges. Whenever the clock on the edge $\langle u,v \rangle$ rings, the 
opinions at $u$ and $v$ jump closer to each other, see Figure \ref{compasspic}.  The parameter $\mu$ determines how much 
they are approaching each other; in the extreme case $\mu=\frac12$, an update on $\langle u,v \rangle$ results in $\eta(u)=\eta(v)$ after the jump. 



\begin{figure}[tbh]
	\centering
	\includegraphics[scale=0.9]{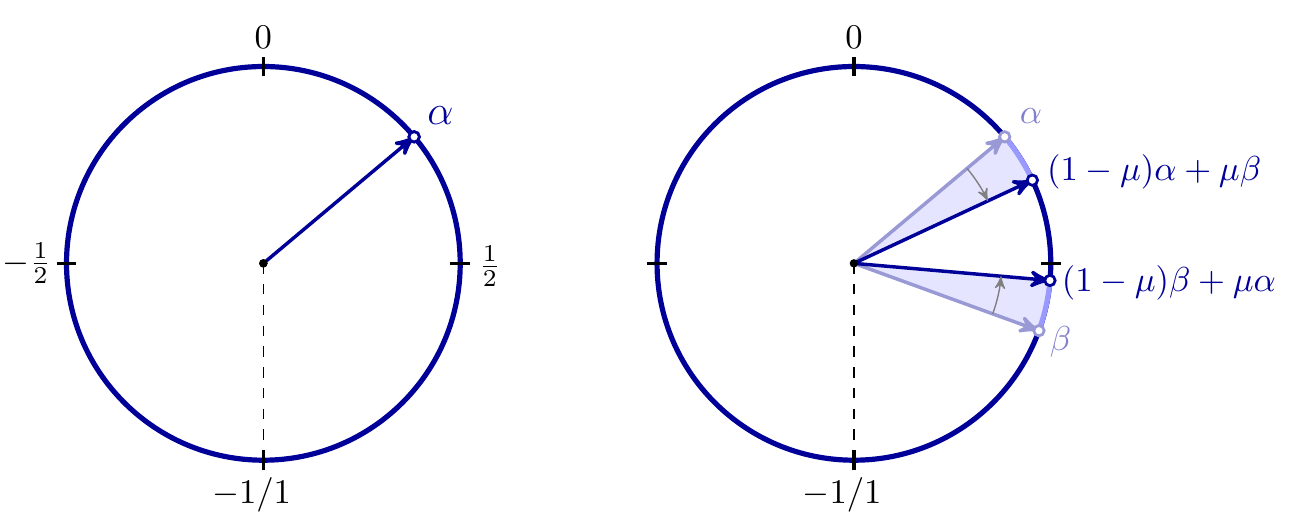}  
	\caption{On the left a visualization of the opinion space and a single opinion (or direction, represented by an
		``angle'' $\alpha$), on the right the effect of an update of two neighboring opinions $\alpha$ and $\beta$.
		\label{compasspic}}
\end{figure}

	The opinion formation model with the unit circle as opinion space, i.i.d.\ $\mathrm{unif}(\mathcal{S})$ initial opinions
	and dynamics with respect to the distance measure $d$, as described in \eqref{eqDefLf2} and \eqref{eqDefAe2}, will be
	referred to as the {\em uniform compass model}.
	
\paragraph{Our results.}
Our main focus is on the long-time behavior of the compass model: Will opinions of neighboring individuals align (`weak consensus')? Will there be global agreement on one direction (`strong consensus')? Our main results answer this question for the uniform compass model on $\mathbb Z$, see Theorem \ref{main}. For the compass model on $\mathbb Z ^n$, $n \geq 2$, we only have a partial answer, see part (c) of Remark \ref{rmk2}.
We start by formalizing these notions. 
\begin{definition}\label{states}
	We distinguish the following three asymptotic regimes:
\begin{enumerate}[(i)]
   \item {\itshape No consensus}\\
   There exist $\epsilon>0$ and two neighbors $\langle u,v\rangle$, s.t.\ for all $t_0\geq0$ there exists $t>t_0$ with
   \begin{equation}
   	d\big(\eta_t(u),\eta_t(v)\big)\geq\epsilon.
   \end{equation}
   \item {\itshape Weak consensus}\\
	Every pair of neighbors $\langle u,v\rangle$ will finally concur, i.e.\ for all $e=\langle u,v\rangle\in E$
	\begin{equation}\label{eqDefWeakConsensus}
		d\big(\eta_t(u),\eta_t(v)\big)\to 0,\text{ as } t\to\infty.
	\end{equation}
   \item {\itshape Strong consensus}\\
	The value at every vertex converges to a common (possibly random) limit $L$, i.e.\
	for all $v\in V$
	\begin{equation}\label{eqDefStrongConsensus}
		d\big(\eta_t(v),L\big)\to 0,\text{ as } t\to\infty.
	\end{equation}
\end{enumerate}
In cases (ii) and (iii), we speak of almost sure consensus / consensus in mean  / consensus in probability
whenever the convergence in \eqref{eqDefWeakConsensus} and \eqref{eqDefStrongConsensus} is almost
surely / in $\mathcal L^1$ / in probability.
\end{definition}

It is a simple exercise to show that on \emph{finite} graphs weak consensus directly implies strong consensus (making both equivalent). 
This, however, is not necessarily true on infinite graphs. 

For the Deffuant model described earlier, only two scenarios have been observed so far: either there is almost
sure strong consensus, or a.s.\ no consensus. We prove that for the compass model, the situation is quite
different, and fairly delicate: 

\begin{theorem}\label{main}
For the compass model on $\Z$ with i.i.d.\ uniform initial distribution, there is weak consensus in mean, but no
strong consensus in probability.
\end{theorem}

We show that the opinions will not converge to one common value (Proposition \ref{notstrong}), although the
pairwise differences of neighboring opinions converge to $0$ in $\mathcal{L}^1$ (Proposition \ref{weak}). Mind that if
there is no strong consensus in probability, there cannot be strong consensus in mean or almost surely.
Further, these results imply that the probability of an individual opinion in the compass model on $\Z$ to converge equals $0$
(Corollary \ref{nolim}).

In Theorem \ref{main}, we start the process from an i.i.d.\ initial configuration (but the independence is lost
immediately). We believe our results to be true for more general initial distributions (cf.\ Remark \ref{rmk1}). On the other hand, they
cannot be true for \emph{all} initial distributions, as there are multiple invariant measures. Indeed, our second
result gives a complete characterization of the invariant measures. To this end, let $\mathcal I$ denote the
set of invariant measures for the generator \eqref{eqDefLf2}. 
Furthermore, for $s\in\mathcal{S}$, denote by $\bar s$ the configuration which assigns the value $s$ to all
vertices, and let $\delta_{\bar s}$ denote the $\delta$-measure which assigns mass 1 to $\bar s$ and 0 to all
other configurations.

\begin{theorem}\label{thm-invariant}
The set $\mathcal I$ of invariant measures for the compass model is given by the convex hull of the set 
\[\big\{\delta_{\bar s};\;  s\in\mathcal{S}\big\}.\]
\end{theorem}
Theorem \ref{main} and Theorem \ref{thm-invariant}, together with the rotational symmetry of our model,
immediately imply the following:

\begin{corollary}\label{weakconv}
For the compass model on $\Z$ with i.i.d.\ uniform initial distribution, the distribution of $\bm{\eta}_t=\big(\eta_t(v)\big)_{v \in \mathbb Z}$
converges weakly to $\int_0^1 \delta_{\bar s}\; ds$ as $t\to\infty$.
\end{corollary}

This means that in a ``typical'' configuration, there will be larger and larger intervals in which the individual opinion values (almost) agree,
but on the other hand, these values will change with time.
To illustrate this phenomenon in the case of discrete opinions, consider the easier (and well-known) voter model on $\mathbb Z$ with an
``interface'', i.e.\ starting with the configuration $\bm{\eta}_0$, where $\eta_0(v) = \mathbbm{1}_{\{v>0\}} $. It is known, see
\cite[Chap.\ V, Thm.\ 1.9]{Liggett}, that then $\bm{\eta}_t=(\eta_t(v))_{v \in \mathbb Z}$ converges weakly to $\frac{1}{2}\delta_{\bar 0} +
\frac{1}{2}\delta_{\bar 1}$.  This means that in any fixed finite interval most likely the vertices either all have the value $0$ or all have the
value $1$, each with probability close to $\frac{1}{2}$ for large $t$. On the other hand, the value of each vertex will change infinitely often
as time progresses. In this Boolean example, one can easily see what causes the phenomenon: the configuration at time $t$ will still have
one edge with all values $0$ to the left and all values $1$ to the right. The only thing changing is the position of this ``interface'' between
$0$'s and $1$'s, which moves as a simple symmetric random walk. Hence, for any finite interval, the probability to see both values at time
$t$ equals the probability that the random walk is in that interval at time $t$, which goes to $0$ (due to the central limit theorem). On the
other hand, since the random walk on $\Z$ is recurrent, the interface will return infinitely often to any given edge and hence each vertex
will change its value infinitely often as time progresses. We believe that something similar happens for the compass model and conjecture in
particular that there is no almost-sure weak consensus, see Section \ref{future}.

\begin{remark}\label{rmk1}
Inspecting the proofs shows that indeed Theorem \ref{main} and Corollary \ref{weakconv} remain valid if we start from a translation invariant,
ergodic sequence $(\eta_0(v))_{v \in \Z}$, having the uniform law $\mathrm{unif}(\mathcal{S})$ as its marginal.
For better readability, we gave the statements and proofs for i.i.d.\ initial opinions.  
\end{remark}  

\paragraph{Comparison with a dynamic XY-model.} 
The compass model has the same state space as the famous XY-model, which is the $O(N)$-model in the
special case $N=2$ (see e.g.\ Chapter 9.1 of \cite{FriedVelen17}). However, the behavior of this model is
rather different from the compass model, as we explain next. 

As for the XY-model on the one-dimensional lattice $\mathbb Z$, it is implicit in the work of McBryan and Spencer
\cite{McbrySpenc77} that the correlations decay exponentially fast, consequently there is a unique Gibbs measure in one dimension. 
Bauerschmidt and Bodineau \cite{BauerBodin18} show that this implies a logarithmic Sobolev inequality for
high temperature, and it may be possible to extend this to low temperature. 
The general criterion of Stroock and Zegarlinski \cite{StrooZegar95} then implies that the Glauber dynamics
of the XY-model on $\mathbb Z$ is ergodic, that is, there is a unique stationary distribution (namely, the
Gibbs measure) and for any starting point, the law at time $t$ converges to this stationary distribution.
This is in sharp contrast to our results for the compass model. 

\paragraph{Organization of the paper.} 
In the next section, we introduce the difference process, which plays a crucial role in the forthcoming sections.
We then turn to the compass model on finite graphs -- 
more precisely paths and rings -- in Section \ref{finite}. In addition to it, we draw a comparison to the trivial
standard Deffuant model (i.e.\ i.i.d.\ $\mathrm{unif}([0,1])$ initial opinions and $\theta=1$) on finite graphs
and highlight the qualitative differences in Subsection \ref{vsDeffuant}.
In Section \ref{nostrong}, we verify that in the uniform compass model on $\Z$, due to its symmetries, there can't
be any form of convergence to a common value. That the differences of neighboring opinions in this setting
converge to 0 (in mean) is established in Section \ref{strictlyweak}, completing the proof of Theorem \ref{main}.
A characterization of invariant measures for the compass model on $\Z$ is given in Section \ref{invariant}. 
We close the paper with a discussion of related open problems in Section \ref{future}.

\section{Preliminaries}

\subsection{The difference process} 
Let us now introduce a slight change of perspective and consider differences between neighbors instead of the plain opinion values;
an approach which will turn out to be more suitable in the context of weak consensus.

\begin{definition}
	Given a configuration of opinions $\bm{\eta}_t=(\eta_t(v))_{v\in V}\in(-1,1]^V$, define the corresponding {\em configuration
	of edge differences} $\bm{\Delta}_t=\big(\Delta_t(e)\big)_{e\in E}$ in the following way: Assign to each edge $e=\langle u,v\rangle$
	the unique value $\Delta_t(e)\in(-1,1]$, such that \[\eta_t(u)+\Delta_t(e)=\eta_t(v) \pmod{\mathcal{S}}.\] See Figure \ref{Delta} for a numerical
	illustration on a section of $\Z$.
\end{definition}

\begin{figure}[bth]
	\centering
	\includegraphics[scale=0.8]{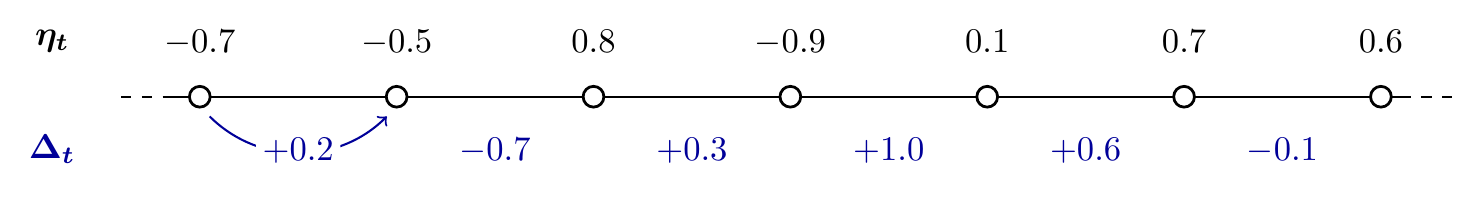}
	\caption{An illustrating example of the transition from plain angles/directions to edge differences in the compass model.\label{Delta}}
\end{figure}
As far as the dynamics is concerned, recall that a Poisson event at time $t$ on the edge $e=\langle u,v\rangle$ 
changes the ($\mathcal{S}$-valued) opinions of the incident vertices $u$ and $v$ by pulling them symmetrically towards their angle bisector
(cf. Figure \ref{compasspic}). 
For the difference process, this corresponds to a $\mu$-fraction of $\Delta_{t-}(e)$ being added to
all edges to which exactly one of $\langle u,v\rangle$ is incident, while at the same time $\Delta_{t-}(e)$ decreases by a factor $1-2\mu$ and no
changes are made on edges to which neither $u$ nor $v$ are incident, i.e.
\begin{equation}\label{gendeltaupdate}
\Delta_t(e')=\begin{cases} (1-2\mu)\,\Delta_{t-}(e),&\text{for }e'=e\\
															\Delta_{t-}(e')+\mu\,\Delta_{t-}(e)\pmod {\mathcal{S}},&\text{for }|e'\cap e|=1\\
															\Delta_{t-}(e'),&\text{for }|e'\cap e|=0.
\end{cases}
\end{equation}
 See Figure \ref{evolution} for an illustration of the dynamics (in the special case of a path).

\subsection{Ergodicity on $\Z$}
A key ingredient in our proofs is the following version of Birkhoff's ergodic theorem.
Let $\bm{\eta}_0=\big(\eta_0(v)\big)_{v\in\Z}$ be the i.i.d.\ sequence of initial opinions and $T$ denote the shift to the left
on $\Z$, i.e.\ $T(v)=v-1$. Given a two-sided sequence $\boldsymbol{X}=(X_v)_{v\in\Z}$, we write $T\boldsymbol{X}$ for the sequence in
which all labels got shifted down by one, i.e.\ the value at $v$ is taken to be $X_{v+1}$ for all $v$.
Further, let $Y_v$ stand for the couple consisting of $\eta_0(v)$ and the Poisson process associated with the edge $\langle v, v+1\rangle$.
Observe that $\boldsymbol{Y}=(Y_v)_{v\in\Z}$ is also an i.i.d.\ sequence and embodies the full randomness of the model. From ergodicity,
we can conclude that the limit of spatial averages almost surely converges to the mean:

\begin{lemma}\label{ergodic}
	Let $\boldsymbol{Y}=(Y_v)_{v\in\Z}$ be as above and $f$ be a real-valued integrable function of $\boldsymbol{Y}$. Further let
	$(\Lambda_n)_{n\in\N}$ be a nested sequence of finite sections of $\Z$ that are strictly increasing in size. Then
	\begin{equation}\label{erglim}
	\lim_{n\to\infty}\frac{1}{|\Lambda_n|}\,\sum_{k\in\Lambda_n} f(T^k\boldsymbol{Y})= \E\big(f(\boldsymbol{Y})\big)\quad\text{a.s.}
	\end{equation}
\end{lemma}

	Bearing in mind that any integrable factor of an i.i.d.\ sequence is ergodic (with respect to the shift $T$, see for
	instance Thm.\ 7.1.3.\ in \cite{Durrett}), the statement is an immediate consequence of Birkhoff's pointwise ergodic theorem
	(see for instance Thm.\ 7.2.1.\ in \cite{Durrett}) adapted to two-sided sequences.\vspace*{1em}

So, if we look at the regimes from Definition \ref{states} from the perspective of pointwise convergence, ergodicity
of the model on $\Z$ (with respect to shifts) ensures that each of the corresponding three events (being translation
invariant) either occurs with probability $0$ or $1$.

\section{Asymptotics on finite graphs}\label{finite}
In order to get acquainted with both the model and some of the arguments/tools, which will be used in the analysis
of the uniform compass model on $\Z$, we start with an investigation of basic finite networks that share some essential
properties with the two-sided infinite path.

Before we turn to finite networks, however, let us make the following two simple observations about the process of
edge differences, which also apply to infinite networks: First, on any tree (i.e.\ cycle-free graph), the properties of the
initial opinion configuration $\bm{\eta}_0$ in the uniform compass model make $\bm{\Delta}_0$ an i.i.d.\ collection of
$\mathrm{unif}\big((-1,1]\big)$ random variables as well. Second, if the maximal degree in the network is $2$, two compromising
agents change the edge difference on at most three edges. As a consequence, for an update on $e=\langle u,v\rangle$ at time
$t$, the following inequality holds:
\begin{equation}\label{ineq}
\sum_{e'\cap\{u,v\}\neq\emptyset}\big|\Delta_t(e')\big|\quad\leq\quad\sum_{e'\cap\{u,v\}\neq\emptyset}\big|\Delta_{t-}(e')\big|.
\end{equation}
To see this, note that $d\big(\eta_t(u),\eta_{t-}(u)\big)=d\big(\eta_t(v),\eta_{t-}(v)\big)=\mu\cdot\big|\Delta_{t-}(e)\big|$, hence
the edge difference on edges incident to exactly one of $u,v$ can not increase by more than that. Since there are at most two
such edges and $\big|\Delta_{t}(e)\big|=(1-2\mu)\cdot\big|\Delta_{t-}(e)\big|$, the claimed inequality follows. Observe at this
point, that \eqref{ineq} can fail whenever $e$ intersects more than 2 other edges (e.g.\ in $\Z^n,\ n\geq 2$).

\subsection{The compass model on paths}\label{section-paths}

As a warm-up, let us analyze the compass model on finite paths $P_n=(V_n,E_n)$ with vertex set $V_n:=\{1,\dots,n\}$ and
edge set $E_n=\{e_1,\dots, e_{n-1}\}$, where $e_v:=\langle v,v+1\rangle$, $v=1,\dots,n-1$.
Here, a Poisson event on $e_v$ will effect only the differences on edges in the set $\{e_{v-1},e_v,e_{v+1}\}$
-- it might be only $\{e_v,e_{v+1}\}$ or $\{e_{v-1},e_v\}$ respectively, in case $e_v$ lies at one end of the path.
More precisely, the update rule for the process of edge differences on a path reads:
\begin{equation}\label{deltaupd}
\left(\begin{array}{c} \Delta_t(e_{v-1})\\\Delta_t(e_v)\\\Delta_t(e_{v+1})\end{array}\right)=
\left(\begin{array}{c} \Delta_{t-}(e_{v-1})+\mu\,\Delta_{t-}(e_v)\\(1-2\mu)\,\Delta_{t-}(e_v)\\\Delta_{t-}(e_{v+1})+\mu\,\Delta_{t-}(e_v)
\end{array}\right)  \pmod {\mathcal{S}}
\end{equation}
and no changes for edges other than $e_{v-1}$, $e_v$ or $e_{v+1}$; see Figure \ref{evolution} for an example.

\begin{figure}[htb]
	\centering
	\includegraphics[scale=0.9]{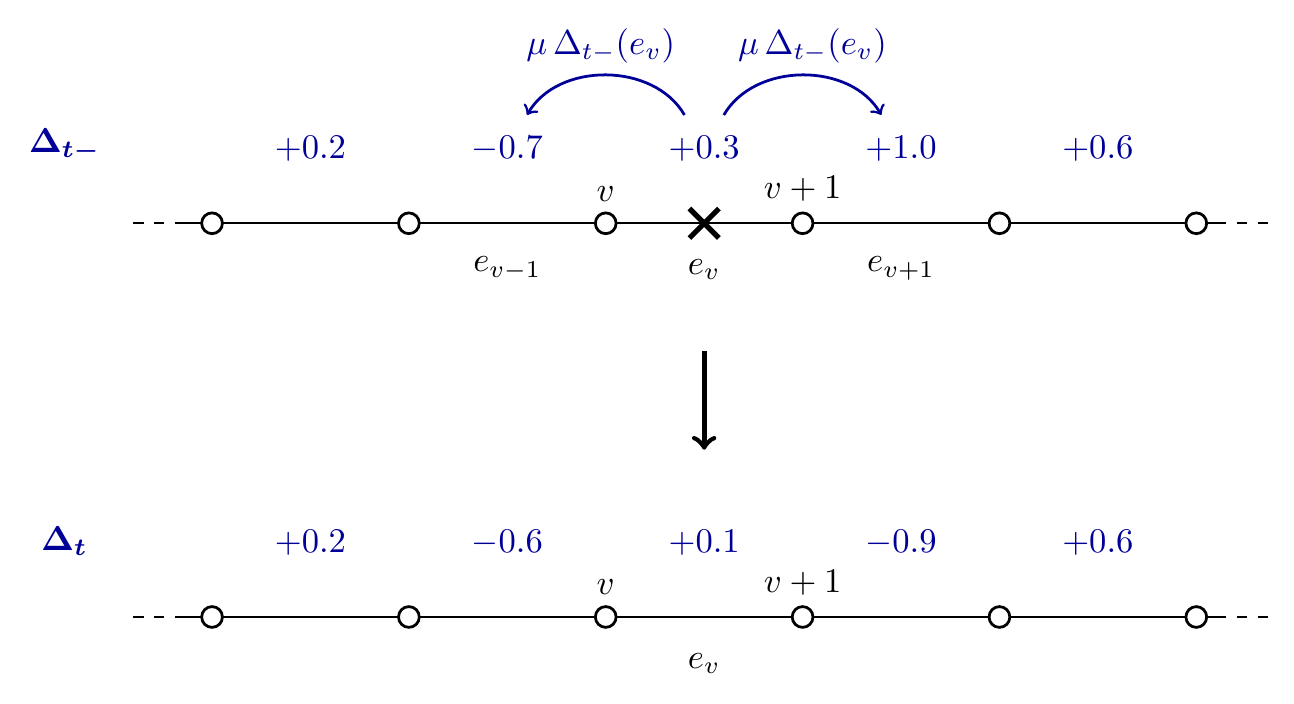}
	\caption{The evolution of the process $(\boldsymbol{\Delta}_t)_{t\geq 0}$, driven by the Poisson events, illustrated by a numerical
		example (here with $\mu=\frac13$) on a path.\label{evolution}}
\end{figure}
Since on $P_n$ the maximal degree is $2$, inequality \eqref{ineq} applies and is sufficient to settle the compass model's
asymptotic behavior: 

\begin{lemma}\label{prepare}
	Fix $n\in\N$ and consider the compass model on the path $P_n$. There will be almost sure weak consensus in the limit, that is, 
	\begin{equation}\label{path_weak}\lim_{t\to\infty}\sum_{e\in E_n} \big|\Delta_t(e)\big|=0\quad \text{a.s.}\end{equation}
\end{lemma}

\begin{proof}
	By \eqref{ineq}, the random variable $W_n(t)=\sum_{e\in E_n} \big|\Delta_t(e)\big|$ is non-increasing
	in $t\geq0$ (and non-negative). For it to converge, the value $\Delta_t(e_1)$ has to converge to $0$ as $t\to\infty$, since
	any update on $e_1=\langle1,2\rangle$ will decrease $W_n(t)$ by at least $\mu\,\big|\Delta_t(e_1)\big|$ -- and due to independence
	of the Poisson processes there will a.s.\ be updates on $e_1$ at arbitrarily large time points. This in turn can only happen if
	$\Delta_t(e_2)$ also converges to 0: For arbitrary $\epsilon>0$, given $\big|\Delta_t(e_1)\big|\leq\epsilon$, any update on $e_2$ will
	increase $\big|\Delta_t(e_1)\big|$ by at least $\mu\,\big|\Delta_t(e_2)\big|-\epsilon$. Iterating this argument proves the claim.
\end{proof}

\vspace*{1em}

Note that by the finiteness of $P_n$ -- as mentioned just after Definition \ref{states} -- Lemma \ref{prepare} in fact proves almost
sure {\em strong} consensus for the compass model on $P_n$ (even irrespective of the initial configuration).

Using Lemma \ref{prepare}, we are further able to conclude that appropriate sequences of updates can produce a flat
configuration on any finite path in the network $G=(V,E)$ in terms of the absolute values of edge differences, uniformly in
the configurations on which they are applied: Let us consider the compass model on $G$, together with a path
$P_n=(V_n, E_n)\subseteq G$ on $n$ nodes and let
\[F_n:=\big\{e=\langle u,v\rangle;\; e\notin E_n, V_n\cap\{ u,v\} \neq\emptyset\big\}\] denote the edge boundary of $P_n$ in $G$.

\begin{corollary}\label{cor}
	Let $P_n$ and $F_n$ be as above and fix $\epsilon,\delta>0$. Then, uniformly in $T\geq 0$, the following
	event has probability $p=p(\epsilon,\delta)>0$: In the time period $(T,T+\delta]$ there will be no Poisson events on $F_n$ and sufficiently
	many on the edges in $E_n$ so that $\sum_{e\in E_n}\big|\Delta_{T+\delta}(e)\big|\leq \epsilon$, irrespectively of the configuration
	$\boldsymbol{\Delta}_T$.
\end{corollary}

\begin{proof}
	To begin with, note that our general assumptions ($G$ is locally finite) ensure the finiteness of $F_n$. Then convince yourself
	of the following three simple facts:
	\begin{enumerate}[(i)]
		\item On a finite collection of edges, with probability 1 there will be only finitely many and no simultaneous Poisson events during a finite
		time period.
		\item By independence of the Poisson processes, for any $T\geq0,\ s>0,\ m\in\N$ and $e^{(k)}\in E_n\cup F_n$, $1\leq k\leq m$,
		the chronologically ordered pattern of locations of all Poisson events on the edges in $E_n\cup F_n$ during the time period
		$(T,T+s]$ has strictly positive probability to be given by the finite sequence $(e^{(1)},\dots,e^{(m)})$.
		\item The time homogeneity of the Poisson processes implies that for every such pattern and fixed $\delta$, the probability to occur in
		$(T,T+\delta]$ is the same for all $T\geq0$.
	\end{enumerate}
	
	From Lemma \ref{prepare} together with facts (i) and (ii), we can deduce that for every configuration $\boldsymbol{\Delta}_T$,
	there exist $m\in \N$ and $(e^{(1)},\dots,e^{(m)})\in (E_n)^m$ such that the following holds: If the chronologically ordered pattern of locations
	of all Poisson events on the edges in $E_n\cup F_n$ during the time period $(T,T+\delta]$ is given by $(e^{(1)},\dots,e^{(m)})$, we end
	up with $\sum_{e\in E_n}\big|\Delta_{T+\delta}(e)\big|\leq \epsilon$.
	
	To verify the claim, we have to find one such pattern which achieves this for all possible $\boldsymbol{\Delta}_T$ at once. By fact
	(iii) we can set $T=0$ without loss of generality.
	Now consider the configuration of all ones, i.e.\ $\boldsymbol{\xi}\in (\R_{\geq0})^E$ given by $\xi(e)=1$, for all $e\in E$. Each Poisson
	event on an edge $e\in E$ at a time $t>0$ will lead to an update of $\boldsymbol{\Delta}_{t}$ according to \eqref{deltaupd}. We will set
	$\boldsymbol{\xi}_0:=\boldsymbol{\xi}$ and update it simultaneously, according to the very same rule
	\eqref{deltaupd}
	but drop the modulo calculation, i.e.\
	\begin{equation*}
	\xi_t(e')=\begin{cases} (1-2\mu)\,\xi_{t-}(e),&\text{for }e'=e\\
	\xi_{t-}(e')+\mu\,\xi_{t-}(e),&\text{for }|e'\cap e|=1\\
	\xi_{t-}(e'),&\text{for }|e'\cap e|=0.
	\end{cases}
	\end{equation*}
    While this makes $\xi_t(e) > 1$ possible, it is not hard to check that for any $e\in E$ and $t\geq 0$, the domination
	\begin{equation}\label{domi}
	\xi_t(e)\geq \big|\Delta_t(e)\big|
	\end{equation}
	holds uniformly in $\boldsymbol{\Delta}_0$ and the sequence of updates.
	As the inequality \eqref{ineq} remains valid with $\boldsymbol{\xi}_t$ in place of $\boldsymbol{\Delta}_t$ (i.e.\ without the modulo
	calculation), the line of reasoning in the proof of Lemma \ref{prepare} applies without any further amendments to $\{\xi_{t}(e);\;e\in E_n\}$
	as well and by \eqref{domi}, the pattern of locations of Poisson events $(e^{(1)},\dots,e^{(m)})\in (E_n)^m$, which achieves
	$\sum_{e\in E_n}\xi_{\delta}(e)\leq \epsilon$ works for all configurations $\boldsymbol{\Delta}_0$ and thus verifies the claim.
\end{proof}
\vspace*{1em}

\subsection{The compass model on rings}
As an extension of Lemma \ref{prepare} and a warm-up for the analysis of the compass model on $\Z$, let us look at the
model on finite rings. Based on the notation of Section \ref{section-paths}, we write $R_n=(V_n,\mathring{E}_n)$ for the ring on
$n$ nodes, with $V_n$ as before and $\mathring{E}_n=E_n\cup\{e_n\}=\{e_1,\dots, e_{n}\}$, where $e_n:=\langle n,1\rangle$,
see Figure \ref{ring}.

\begin{figure}[htb]
	\centering
	\includegraphics[scale=1]{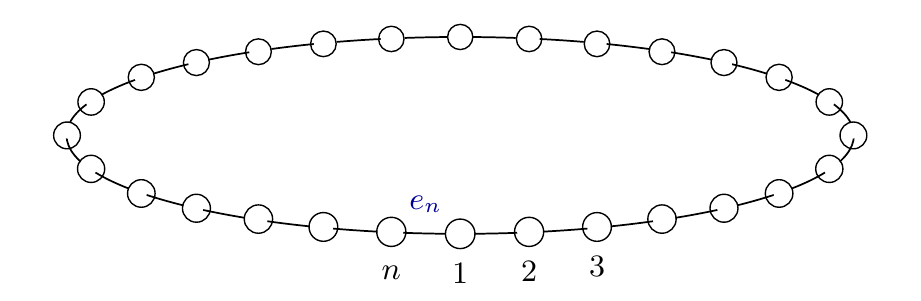}
	\caption{In this subsection, we consider a finite ring as underlying network graph.\label{ring}}
\end{figure}

\begin{proposition}
	Fix $n\in\N$ and consider the compass model on the ring $R_n$. There will be almost sure strong consensus in the limit.
\end{proposition}

\begin{proof}
As long as there is no Poisson event on the edge $e_n$, the compass model on $R_n$ behaves exactly like the model on
$P_n$. From Lemma \ref{prepare}, we know that in this setting $\sum_{e\in E_n} \big|\Delta_t(e)\big|$
converges to $0$ almost surely.

In fact, Corollary \ref{cor} is the key ingredient for the remainder of the proof.
Choose $\epsilon>0$ and let $A_t$ be the event that during the time period $(t,t+1]$ there are no Poisson events on $e_n$
and sufficiently many on the edges in $E_n$ such that $\sum_{e\in E_n}\big|\Delta_{t+1}(e)\big|\leq \epsilon$, irrespectively of the
configuration $\boldsymbol{\Delta}_t$. Applying the corollary with $G=R_n$, hence $F_n=\{e_n\}$, and $\delta=1$, we
are guaranteed a number $p>0$, such that $\Prob(A_t)=p$ for all $t\geq 0$.

At this point, the following three observations are crucial:
First, by the triangle inequality it trivially holds that $\big|\Delta_t(e_n)\big|\leq\sum_{e\in E_n}\big|\Delta_t(e)\big|$.
Second, $W(t)=\sum_{e\in\mathring{E}_n} \big|\Delta_t(e)\big|$ is non-increasing by  \eqref{ineq} and
third, the events $(A_k)_{k\in\N}$ are independent by the memoryless property of the Poisson processes.
If we now use the sequence $(A_k)_{k\in\N}$ to define a random variable $Y$ by letting $Y(\omega)=k$ whenever
$\omega\in A_k\setminus \bigcup_{j=1}^{k-1} A_j$,  for all $k\in\N$, then $\{W(k)>2\epsilon\}\subseteq\{Y>k\}$ and $Y$ is
geometrically distributed with parameter $p$. We conclude that 
\[\Prob\Big(\lim_{t\to\infty}W(t)\leq 2\epsilon\Big)=1\] and hence
almost sure weak consensus. As before, by the finiteness of the network, this directly implies a.s.\ strong consensus
and thus proves the claim.
\end{proof}

\subsection{Compass vs.\ Deffuant model}\label{vsDeffuant}
In this subsection, we want to compare the asymptotic behavior of the compass model with the one of the trivial ($\theta=1$)
standard Deffuant model -- as mentioned in the introduction, the latter has in principle the same dynamics (compare \eqref{eqDefAe}
and \eqref{eqDefAe2}), however, with the interval $[0,1]$ a convex opinion space.

It is not hard to see that on $P_n$ and $R_n$, the standard Deffuant model with trivial confidence parameter $\theta$ exhibits
the same asymptotics (a.s.\ strong consensus) -- in fact by the very same arguments. Nevertheless, there are qualitative
differences in terms of randomness and distribution of the limiting variable $L$.

Due to the fact that there is no modulo operation involved, the dynamics of the Deffuant model preserves
the sum of updated opinions. For this reason, on a finite graph, the initial opinions already determine the final consensus
value, simply being their average.

Let us, for the sake of simplicity, go back to $P_n=(V_n,E_n)$, the path on $n$ nodes, and
illustrate the qualitative differences between compass and trivial standard Deffuant model with help of the following example:
Start with an i.i.d.\ uniform initial configuration $\big(\eta_0(v)\big)_{v\in V_n}$, to be more precise: with $\mathrm{unif}(\mathcal{S})$
as marginal for the compass and $\mathrm{unif}([0,1])$ as marginal for the trivial Deffuant model.
As derived above, in both models we observe almost sure strong consensus in the limit. However, while the common final
value $L=\lim_{t\to\infty}\eta_t(v)$ in the trivial Deffuant model equals $L_\mathrm{D}(P_n)=\frac1n\,\sum_{v\in V_n}\eta_0(v)$
(and hence does not depend on the dynamics), the modulo operation in the compass model (with fixed starting configuration)
produces in the limit $t\to\infty$ a value
\[L_\mathrm{c}(P_n)=\frac1n\,\Big[2K+\sum_{v\in V_n}\eta_0(v)\Big],\]
where $K$ is an integer-valued random variable, depending on the sequence of updates (and in fact also the initial values).
It is easy to see that given $\big(\eta_t(v)\big)_{v\in V_n}$, the common limit value $L_\mathrm{c}(P_n)$ can depend on the
future dynamics only if $\{\eta_t(v);\;v\in V_n\}$ is not yet contained in a half-circle, more precisely a connected part of $\mathcal{S}$
containing exactly one of each pair of diametrically opposed opinion values.

Furthermore, in the limit of longer and longer paths ($n\to\infty$), the strong law of large numbers dictates that
$L_\mathrm{D}(P_n)$ converges to $\frac12$ almost surely (i.e.\ becomes degenerate), while $L_\mathrm{c}(P_n)$ is a
$\mathrm{unif}(\mathcal{S})$ random variable for all $n$, caused by the rotational symmetry in the opinion space of the compass model.
\vspace{1em}

Finally, in contrast to the Deffuant model, the compass model is {\em noise-sensitive} in the following sense: 
Let us couple two copies of the compass model on $P_n$ by starting from two initial configurations, $\big(\eta_0(v)\big)_{v\in V_n}$
and $\big(\eta'_0(v)\big)_{v\in V_n}$ respectively, which disagree only at one site, i.e.\ there exists $v\in V_n$ s.t.\
$\eta'_0(u)=\eta_0(u)$ for all $u\in V_n\setminus\{v\}$, and further taking the very same i.i.d.\ Poisson processes to drive the
dynamics. Let $L_\mathrm{c}(P_n),\ L'_\mathrm{c}(P_n)$ denote the corresponding limit values of
both copies.
While in the trivial Deffuant model, altering one single initial opinion can change the common limit by at most $\frac1n$, the
two limits $L_\mathrm{c}(P_n)$ and $L'_\mathrm{c}(P_n)$ can be at distance $1$, which is the maximal possible value as
$d(x,y)\leq1$ for all $x,y\in \mathcal{S}$. To see why this is true, let us sketch a numerical example:

\begin{example}
Take $P_{2n-1}$ and let
$\eta_0(v)=\frac{v}{n}-1$, for all $v\in V_{2n-1}=\{1,\dots,2n-1\}$. To get to $\big(\eta'_0(v)\big)_{v\in V_n}$, we only replace $\eta_0(n)=0$ by $\eta'_0(n)=1$. See Figure \ref{butterfly} for an illustration. 
\begin{figure}[htb]
	\centering
	\includegraphics[scale=0.9]{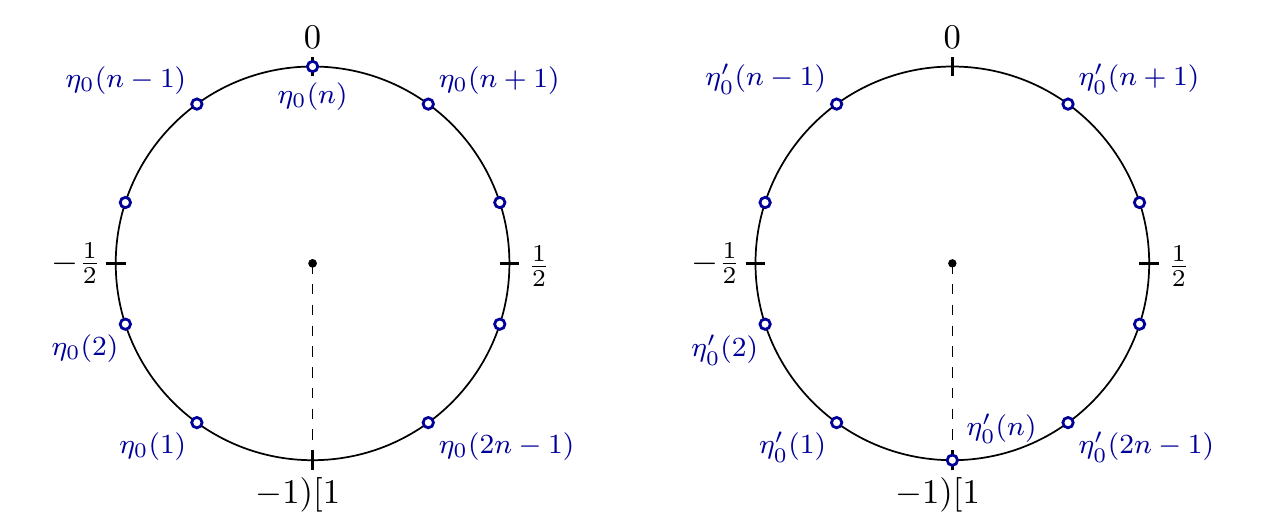}
	\caption{Two almost identical starting configurations that demonstrate the noise-sensitivity of the compass model.\label{butterfly}}
\end{figure}
If up to some large time $T$ there are no updates involving site $n$, but plenty of Poisson events on all other edges
$e\in E_n\setminus\{e_{n-1},e_n\}$, both of the configurations will see the opinions at sites $1$ through $n-1$ gather around the
value $-\frac12$ and opinions at sites $n+1$ through $2n-1$ gather around the value $\frac12$. If after $T$ the Poisson
events on $e_{n-1}$ and $e_n$ are somewhat alternating, i.e.\ not too many updates on one of both during a time period that
does not see any update on the other, it will lead to $L_\mathrm{c}(P_n)=0$ and $L'_\mathrm{c}(P_n)=1$.

It is further not so hard to come up with an example, in which even a slight change of one value can cause this kind of
butterfly effect.
\end{example}

Let us now leave finite paths and focus on the case of $G$ being the one-dimensional integer lattice $\Z$. As far as the
trivial standard Deffuant model is concerned, the asymptotic behavior actually remains almost sure strong consensus
(cf.\ Thm.\ 1.4 in \cite{Lanchier} or Thm.\ 6.5 in \cite{ShareDrink}). In higher dimensions (i.e.\ $\Z^d$, $d\geq2$) even for
trivial bounded confidence parameter (i.e.\ $\theta=1$) so far only almost sure weak consensus could be verified
(cf.\ Thm.\ 3.1 in \cite{Deffuant}). Nonetheless, for the trivial Deffuant model it is believed that even on $\Z^d$, $d\geq2$,
almost sure {\em strong} consensus is the right answer and the step from weak to strong consensus more a technical
cumbersomeness, which has to be taken care of.

The different topology in the opinion space of the compass model, however, renders a central energy argument  (cf.\ Lemma 3.2
in \cite{Deffuant}) void and in some sense opens a door to qualitatively different asymptotics. As we will see in the
subsequent sections, the behavior of the uniform compass model on $\Z$ (in the limit as $t\to\infty$) is indeed {\em strictly weak}
consensus (in mean). Note at this point that simulation studies are rather not a suitable tool to tell apart strong and weak
consensus, since strictly weak consensus cannot appear on finite graphs, as remarked earlier.

\section{No strong consensus in the uniform compass model on \texorpdfstring{$\Z$}{\bf Z}}\label{nostrong}

For the remainder of this paper, we analyze the compass model on $\Z$ with i.i.d.\ $\mathrm{unif}(\mathcal{S})$ initial opinions.
In this section, we show that the symmetries of the uniform compass model rule out strong consensus (in any sense).

\begin{proposition}\label{notstrong}
	For the uniform compass model on $\Z$, there is no strong consensus in the limit (not even in probability).
\end{proposition}

\begin{proof}
	This result readily follows from the symmetries and invariances of the model.
	Let us first rule out almost sure strong consensus and assume for contradiction that there exists a $(-1,1]$-valued
	random variable $L$ for which \eqref{eqDefStrongConsensus} holds a.s. Consequently,
	$$B=\Big\{\lim_{t\to\infty}d\big(\eta_{t}(v),L\big)=0,\text{ for all } v\in\Z\Big\}$$
	is an almost sure event and either $B\cap\{L\in(-1,0]\}$ or  $B\cap\{L\in(0,1]\}$ has probability at least $\frac12$.
	As the uniform initial opinions entail a complete rotational symmetry in $\mathcal{S}$, we can in fact conclude that
	these probabilities coincide, i.e.\ $\Prob(B\cap\{L\in(-1,0]\})=\Prob(B\cap\{L\in(0,1]\})=\frac12$.
	
	Finally, the event $B\cap\{L\in[0,1)\}$ is invariant with respect to shifts on $\Z$, thus forced to either have probability $0$
	or $1$, due to ergodicity of the model: Take $f=\mathbbm{1}_{B\cap\{L\in[0,1)\}}$ in Lemma \ref{ergodic}, which makes the left
	hand side of \eqref{erglim} either have value 0 or 1, depending on $f(\boldsymbol{Y})$, but not $n$.
	From this contradiction it follows that there is no random variable $L$ fulfilling \eqref{eqDefStrongConsensus} almost surely.	
	
	It remains to verify that assuming the existence of a random variable $L$ such that only
    \begin{equation}\label{scp}
    d\big(\eta_t(v),L\big)\stackrel{\Prob}{\longrightarrow}0,\text{ as } t\to\infty,\text{ for all } v\in \Z
    \end{equation}
   holds, similarly leads to a contradiction.

For ease of notation, let us relabel the vertices of $\Z$ to form a one-sided infinite sequence, for example by means of the
standard enumeration \[v_1=0,\ v_{2m}=m\quad\text{and}\quad v_{2m+1}=-m\quad\text{for all }m\in\N.\]
By the subsequence criterion (cf.\ for instance Thm.\ 20.5 in \cite{billingsley}), we can deduce from \eqref{scp} that there
exists a subsequence of $\big(\eta_k(v_1)\big)_{k\in \N}$, say $\big(\eta_{k^{(1)}_{j}}(v_1)\big)_{j\in \N}$, such that
$d\big(\eta_{k^{(1)}_j}(v_1),L\big)$ converges almost surely to $0$ as $j\to\infty$. By the same token, we can choose a
subsequence $\big(k^{(2)}_{j}\big)_{j\in \N}$ of $\big(k^{(1)}_{j}\big)_{j\in \N}$ such that $d\big(\eta_{k^{(2)}_j}(v_2),L\big)$
converges almost surely to $0$ as well. Now iterate this thinning and use Cantor's diagonal argument:
Setting $t_j:=k^{(j)}_{j}$, we accomplished that $\big(\eta_{t_j}(v_m)\big)_{j\in\N}$ is a subsequence of
$\big(\eta_{k^{(m)}_{j}}(v_m)\big)_{j\in\N}$ for all $m\in\N$ (apart from finitely many elements in the beginning) and
consequently \[d\big(\eta_{t_j}(v),L\big)\stackrel{\mathrm{a.s.}}{\longrightarrow}0,\text{ as } j\to\infty,\text{ for all } v\in \Z.\]
Taking $B=\Big\{\lim_{j\to\infty}d\big(\eta_{t_j}(v),L\big)=0,\text{ for all } v\in\Z\Big\}$, the reasoning used in the almost sure
case above still applies and hence the claim follows.	
\end{proof}

\vspace{1em}
\begin{remark}\label{rmk2}
	\begin{enumerate}[(a)]
		\item The rotational symmetry of the model and its initial configuration implies
		$\mathcal{L}\big(\eta_t(v)\big)=\mathrm{unif}(\mathcal{S})$ for all $v\in \Z$ and all times $t>0$. The independence property
		of $\big(\eta_0(v)\big)_{v\in\Z}$ is, however, lost immediately. 
		The fact that $\eta_t(0)$ has a uniform distribution on $\mathcal{S}$ for all $t$ implies that the marginals of any possible
		(weak) limit must be uniform as well.
		\item Further, observe that the proof of Proposition \ref{notstrong} is based on the symmetries and invariances of the
		uniform compass model only. If one introduces -- in analogy to the non-trivial Deffuant model -- a confidence bound
		$\theta\in(0,1)$, such that only opinions at distance at most $\theta$ will symmetrically approach each	other in an
		update, the line of reasoning above still applies, and consequently Proposition \ref{notstrong} holds just as well for a
		uniform compass model with bounded confidence.
		\item Finally, since $\Z$ is a subgraph of $\Z^n$ for all $n\geq2$, the above proof immediately transfers to
		 higher dimensions, i.e.\ the statement of Proposition \ref{notstrong} holds true for the uniform compass model on
		 $\Z^n, n\geq2$.
	\end{enumerate}
\end{remark}

\section{A case of strictly weak consensus}\label{strictlyweak}

In view of Definition \ref{states} and Proposition \ref{notstrong}, the behavior of the uniform compass model on $\Z$ in the limit
either has to be no consensus or a form of weak consensus. 
We establish the latter: 

\begin{proposition}\label{weak}
	The compass model on $\Z$ with uniform initial opinions exhibits weak consensus in mean.
\end{proposition}

To see, how Corollary \ref{cor} comes in useful here, imagine the following scenario: During a given time interval, the agents on
a fixed finite section of $\Z$ interact a lot, while there is no interaction with the two neighboring ones left and right of this section.
Albeit rarely, this scenario will occur, vacate the corresponding section in terms of the absolute values of edge differences 
(irrespectively of the configuration before) and as a result enable us to establish weak consensus in mean.

It should be mentioned that for a fixed edge $e_u=\langle u,u+1\rangle$, the value of $\Delta_t(e_u)$ matches $d\big(\eta_t(u),\eta_t(u+1)\big)$,
apart from the fact that it additionally carries the sign (clockwise ($+$) or counterclockwise ($-$)) of the smallest angle formed by the
directions represented by $\eta_t(u)$ and $\eta_t(u+1)$. Bearing $d\big(\eta_t(u),\eta_t(u+1)\big)=\big|\Delta_t(e_u)\big|$ in mind, weak consensus
is equivalent to the corresponding componentwise convergence of $\boldsymbol{\Delta}_t$ to $0$. 

As a final preparation for the proof of Proposition \ref{weak}, let us verify that the expected value $\E\big|\Delta_t(e)\big|$, which by symmetry
coincides for all $e\in E$, does not increase with $t$.

\begin{lemma}\label{monotone}
The function $t\mapsto \E\big|\Delta_t(e)\big|,\ t\in[0,\infty)$ is non-increasing.
\end{lemma}

\begin{proof}
To begin with, recall that a Poisson event on $e_v$ at time $t$ can change the $\Delta$-values on the edges $e_v, e_{v-1}$ and
$e_{v+1}$ only and we further have 
\begin{equation*}
\big|\Delta_t(e_{v-1})\big|+\big|\Delta_t(e_v)\big|+\big|\Delta_t(e_{v+1})\big|\leq \big|\Delta_{t-}(e_{v-1})\big|+\big|\Delta_{t-}(e_v)\big|
+\big|\Delta_{t-}(e_{v+1})\big|
\end{equation*}
by \eqref{ineq}. 
For any $t\geq 0$, we can take $f(\boldsymbol{Y})=\big|\Delta_t(e_0)\big|$ in Lemma \ref{ergodic} to get
\begin{equation}\label{2.1adapt}
\lim_{i,j\to\infty} \frac{1}{i+j}\,\sum_{v=-i}^{j-1} \big|\Delta_t(e_v)\big|= \E\big|\Delta_t(e_0)\big|\quad\text{a.s.}
\end{equation}
Next, we can conclude from the independence of the Poisson processes associated to the edges that for all $t,\epsilon\geq 0$, there will
a.s.\ be two strictly increasing sequences of natural numbers, say $(i_n)_{n\in\N}$ and $(j_n)_{n\in\N}$, with the property that neither of
the edges incident to a vertex $v\in\{-i_n, j_n;\; n\in\N\}$ has seen a Poisson event in the time interval $(t,t+\epsilon]$.

This choice ensures that for each $n\in\N$, the average edge difference on the section $\{-i_n,\dots, j_n\}$, i.e.\ 
$\frac{1}{i_n+j_n}\,\sum_{v=-i_n}^{j_n-1} \big|\Delta_s(e_v)\big|$, can change only at times $s\in(t,t+\epsilon]$, which
involve a Poisson event on the section between the vertices $-i_n$ and $j_n$ (in fact $-i_n+1$ and $j_n-1$), and further that it must be
non-increasing as a consequence of the above inequality \eqref{ineq}.
Together with \eqref{2.1adapt}, this establishes the claimed monotonicity of $\E\big|\Delta_t(e)\big|$.
\end{proof}

\vspace*{1em}
Note that the statement of Lemma \ref{monotone} is not limited to the uniform case, but holds for the compass model in
general, i.e.\ for initial marginal distributions other than $\mathrm{unif}(\mathcal{S})$.\vspace*{1em}

\begin{nproof}{of Proposition \ref{weak}}
Let us assume for contradiction that some $\big(\Delta_t(e)\big)_{t\geq0}$ does not converge to 0 in mean as $t\to\infty$.
Due to stationarity, we can assume $e=e_0$ without loss of generality.

Our assumption (together with Lemma \ref{monotone}) implies
\begin{equation}\label{wrongass}
\tfrac12=\E\big|\Delta_0(e)\big|\geq\E\big|\Delta_t(e)\big|\geq\lim_{s\to\infty} \E\big|\Delta_s(e)\big|=\epsilon,
\end{equation}
for some $\epsilon>0$ and all $t\geq 0$.
We will lead this to a contradiction by showing that given \eqref{wrongass}, the difference $\E\big|\Delta_t(e)\big|-\E\big|\Delta_{t+1}(e)\big|$
is bounded away from $0$ (uniformly in $t$), thus forcing $\lim_{s\to\infty} \E\big|\Delta_s(e)\big|=-\infty$.

To achieve this, we fix $t\geq0$, set $K=\big\lceil\frac6\epsilon\big\rceil$ and do the following construction: Partition the edges of $\Z$ into
disjoint blocks of length $K$, i.e.\ paths $(P^{(j)}_K)_{j\in\Z}$, such that $P^{(j)}_K$ connects the vertices $jK$ and $(j+1)K$, for all $j\in \Z$.

Next, observe that $\frac{1}{K}\,\sum_{v=0}^{K-1} \big|\Delta_t(e_v)\big|$ is a $[0,1]$-valued random variable with expectation
$\E\big|\Delta_t(e)\big|\geq\epsilon$. Therefore, it must hold (uniformly in $t$) that
\begin{equation}\label{posprob}
\Prob\bigg(\frac{1}{K}\,\sum_{v=0}^{K-1} \big|\Delta_t(e_v)\big|\geq\frac{\epsilon}{2}\bigg)\geq\frac\epsilon 2.
\end{equation}

Further, from Corollary \ref{cor} we get that the following event, for which we will write $A^{(0)}_t$,
has positive probability, say $p:=\Prob(A^{(0)}_t)>0$:
In the time interval $[t,t+1]$, there are no Poisson events neither on $e_0$ nor on $e_{K-1}$ and sufficiently many on
the edges in $E_{K-1}=\{e_1,\dots,e_{K-2}\}$ such that $\sum_{e\in E_{K-1}} \big|\Delta_{t+1}(e)\big|\leq \frac12$, irrespectively of the configuration $\bm{\Delta}_t$.

For all $j\in\Z$, let us write $A^{(j)}_t$ for the event $A^{(0)}_t$ shifted by $jK$ edges and 
$$B^{(j)}_t:=\bigg\{\frac{1}{K}\,\sum_{v=jK}^{(j+1)K-1} \big|\Delta_t(e_v)\big|\geq\frac{\epsilon}{2}\bigg\}.$$
From stationarity and \eqref{posprob}, it follows that $\Prob(B^{(j)}_t)=\Prob(B^{(0)}_t)\geq\frac\epsilon2$.
Due to the memoryless property of the Poisson processes and the fact that they are independent from $\boldsymbol{\Delta}_0$, or rather
$\bm{\eta}_0$, for each $j\in\Z$ the events $A^{(j)}_t$ (depending on the Poisson events in the time interval $(t,t+1]$ only) and $B^{(j)}_t$
(depending on the start values and dynamics up to time $t$) are independent. Consequently, $A^{(j)}_t\cap B^{(j)}_t$ has probability at least
$\frac{p\epsilon}{2}$. Since the Poisson processes are time homogeneous and the lower bound on $\Prob(B^{(j)}_t)$ is uniform in $t$, the
same actually holds for all $t\geq 0$.

To conclude, we gather a few simple observations: First, given the event $A^{(j)}_t$, it holds that
$$\sum_{v=jK}^{(j+1)K-1} \big|\Delta_{t+1}(e_v)\big|\leq 2+\sum_{v=jK+1}^{(j+1)K-2} \big|\Delta_{t+1}(e_v)\big|\leq \frac52,$$
as $\big|\Delta_{t+1}(e_{jK})\big|$ and $\big|\Delta_{t+1}(e_{(j+1)K-1})\big|$ are trivially bounded by $1$.
Second, given $B^{(j)}_t$, we have a reverse inequality for the time point $t$; more precisely, by our choice of $K$:
$$\sum_{v=jK}^{(j+1)K-1} \big|\Delta_{t}(e_v)\big|\geq K\cdot\frac\epsilon2 \geq 3.$$
In other words, given $A^{(j)}_t\cap B^{(j)}_t$, the sum $\sum_{v=jK}^{(j+1)K-1} \big|\Delta_{t+s}(e_v)\big|$ decreases by at least
$\frac12$ as $s$ increases from $0$ to $1$.

Let $S$ denote a section between two blocks (indexed by $j$ and $k$) such that $A^{(j)}_t\cap B^{(j)}_t$ and $A^{(k)}_t\cap B^{(k)}_t$ hold,
but not for any block in $S$. Now it is crucial to notice that the sum $\sum_{e\in S}\big|\Delta_{t}(e)\big|$ is non-increasing until time $t+1$:
Let $\underline{e}$ and $\overline{e}$ be the two edges of $P^{(j)}_K$ and $P^{(k)}_K$ respectively sharing a vertex
with an edge in $S$. Since there are no Poisson events on $\underline{e}$ and $\overline{e}$ during $(t,t+1]$, no Poisson event outside of
$S$ can change the $\Delta$-values inside $S$ during this time period. According to \eqref{ineq}, events inside $S$ can only decrease the
sum and the claimed monotonicity follows.
The fact that Poisson events on the marginal edges in $S$ might cause $\big|\Delta_{t+1}(\underline{e})\big|>\big|\Delta_{t}(\underline{e})\big|$
or $\big|\Delta_{t+1}(\overline{e})\big|>\big|\Delta_{t}(\overline{e})\big|$ doesn't have to bother us, since we estimated the values on these
edges with the utterly crude upper bound $1$ anyway.

Finally, applying Lemma \ref{ergodic} one last time, taking $T^K$ instead of $T$ and $f(\boldsymbol{Y})=\mathbbm{1}_{A^{(0)}_t\cap B^{(0)}_t}$, we get
\begin{equation}\label{indic}
\lim_{i,j\to\infty} \frac{1}{i+j}\,\sum_{k=-i}^{j-1} \mathbbm{1}_{A^{(k)}_t\cap B^{(k)}_t}= \Prob\big(A^{(0)}_t\cap B^{(0)}_t\big)\geq \frac{p\epsilon}{2} \quad\text{a.s.}
\end{equation}
Choosing instead $f(\boldsymbol{Y})=\frac1K \sum_{v=0}^{K-1} \big|\Delta_{t+s}(e_v)\big|$ for $s\in\{0,1\}$, gives
$$\lim_{i,j\to\infty} \frac{1}{(i+j)K}\,\sum_{v=-iK}^{jK-1} \big|\Delta_{t+s}(e_v)\big|= \E\big|\Delta_{t+s}(e_0)\big|\quad\text{a.s.}$$
From the above, in particular \eqref{indic}, we can deduce the following inequality:
\begin{align*}
\E\big|\Delta_{t}(e_0)\big|-\E\big|\Delta_{t+1}(e_0)\big|
&=\lim_{i,j\to\infty} \frac{1}{(i+j)K}\,\sum_{v=-iK}^{jK-1}\Big[ \big|\Delta_{t}(e_v)\big|-\big|\Delta_{t+1}(e_v)\big|\Big]\\
&\geq\lim_{i,j\to\infty} \frac{1}{(i+j)K}\,\sum_{k=-i}^{j-1} \frac12\cdot\mathbbm{1}_{A^{(k)}_t\cap B^{(k)}_t}\\
&\geq\frac{p\epsilon}{4K},
\end{align*}
where the equality and second inequality hold almost surely. Since $p$ depends on $K$ only, this bound is uniform in $t$; we arrive at the
contradiction sketched above and have thus ruled out the initial assumption. 
\end{nproof}


\begin{nproof}{of Theorem \ref{main}}
This follows from Proposition \ref{notstrong} and Proposition \ref{weak}. 
\end{nproof}

Next, we observe that Proposition \ref{weak}, together with the symmetries of the uniform compass model, renders convergence of individual
opinions impossible:

\begin{corollary}\label{nolim}
Consider the uniform compass model on $\Z$. For any fixed vertex $v\in\Z$, 
$$\Prob\Big(\lim_{t\to\infty}\eta_t(v)\text{ exists}\Big)=0.$$
\end{corollary}

\begin{proof}
Let us write $A:=\{\lim_{t\to\infty}\eta_t(v)\text{ exists}\}$ and assume $\Prob(A)=:p>0$ for contradiction. From the rotational symmetry in $\mathcal{S}$,
it follows that
$$\Prob\Big[\lim_{t\to\infty}\eta_t(v)\in (-\tfrac12,0]\,\Big|\,A\Big]=\Prob\Big[\lim_{t\to\infty}\eta_t(v)\in (\tfrac12,1]\,\Big|\,A\Big]=\frac14.$$

By ergodicity, cf.\ Lemma \ref{ergodic}, the density of nodes at which the opinion converges to a value in $(-\tfrac12,0]$ and $(\tfrac12,1]$ respectively, therefore a.s.\ equals $\frac{p}{4}$.
Hence, for big enough $D\in \N$, with probability at least $\frac12$ there exist two nodes $u,w\in\{0,\dots,D\}$ such that both
$\lim_{t\to\infty}\eta_t(u)$ and $\lim_{t\to\infty}\eta_t(w)$ exist and the former lies in $(-\tfrac12,0]$, the latter in $(\tfrac12,1]$. This, however, forces
$$\sum_{v=0}^{D-1} \big|\Delta_{t}(e_v)\big|\geq \frac12,\quad\text{for all }t \text{ large enough,}$$ contradicting Proposition \ref{weak}.
\end{proof}

\section{Invariant measures}\label{invariant}

In this section, we finally prove Theorem \ref{thm-invariant}.  
Trivially, constant profiles, i.e.\ $\eta(v)=s$ for all $v\in\Z$ and some $s\in(-1,1]$, are invariant under the dynamics of the compass model.
We prove now that these are the only  extremal invariant distributions.

%
%
%

To this end, we first establish that for invariant measures, the edge differences on neighboring edges must have the same sign a.s.\ at all times. 

\begin{lemma}\label{samesign}
Consider an invariant measure $\nu$ for the compass model on $\Z$. Given that we start with $\bm{\eta}_0\sim \nu$ as initial configuration,
 it is true that at any time $t\geq0$, 
\begin{equation}\label{samesigneq}\Delta_t(e_v)\cdot\Delta_t(e_{v+1})\geq 0\quad\text{a.s.\ for all }v\in\Z.\end{equation}
\end{lemma}

\begin{proof}
Let us start by calculating the probabilities for infinitesimal changes of the difference on a given edge, say $e_0$.
To this end, we consider the section $P=\{e_{-2},e_{-1},e_0,e_1,e_2\}$ -- as depicted in Figure \ref{pathof5} -- and write $N_e(t)$ for the Poisson
process associated with edge $e\in E$. Let $A_\emptyset(t)$ denote the event that no updates occur
neither on $e_0$, nor $e_{-1}$, nor $e_1$ until time $t$, i.e.\
\[A_\emptyset(t):=\big\{N_{e_{-1}}(t)=N_{e_0}(t)=N_{e_1}(t)=0\big\}.\]
For $e\in P$, let $A_{e}(t):=\big\{\sum_{e'\in P} N_{e'}(t)=1=N_{e}(t)\big\}$ be the event that until $t$, there is exactly one Poisson event on
$P$, occurring on edge $e$. Finally, let $A_{\geq 2}(t)$ denote the event that there are at least 2 Poisson events on $P$ until time $t$.
    \begin{figure}[htb]
    	\centering
    	\includegraphics[scale=1]{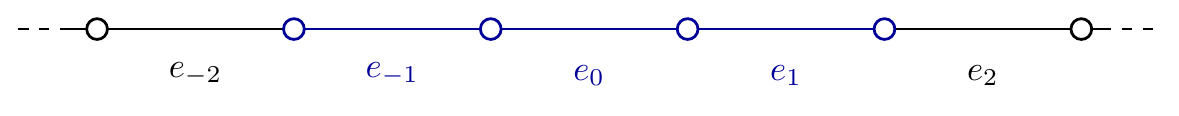}
    	\caption{To understand the infinitesimal evolution of the edge difference $\Delta_t(e_0)$, essentially only updates on $e_0$ and its
    		neighboring edges matter.\label{pathof5}}
    \end{figure}

\noindent
For a Poisson process $\big(N(t)\big)_{t\geq0}$ of unit rate, starting with $N(0)=0$, it holds
\begin{equation}\label{Poi}\Prob\big(N(t)=n\big)=\frac{t^n}{n!}\,\mathrm{e}^{-t},\quad \text{for all }n\in\N_0.\end{equation}
Simple calculations -- based on \eqref{Poi} and the independence of Poisson processes associated with different edges --
    yield
    \begin{equation}\label{calc}
    \begin{split}
     &\Prob\big(A_\emptyset(\epsilon)\big)=\mathrm{e}^{-3\epsilon}=1-3\epsilon+O(\epsilon^2),\\
     &\Prob\big(A_{e_0}(\epsilon)\big)=\Prob\big(A_{e_{-1}}(\epsilon)\big)=\Prob\big(A_{e_1}(\epsilon)\big)
     =\epsilon\cdot\mathrm{e}^{-5\epsilon}=\epsilon+O(\epsilon^2)\quad\text{and}\\
     &\Prob\big(A_{\geq2}(\epsilon)\big)=O(\epsilon^2).
    \end{split}
    \end{equation}
    Consequently, we find (cf.\ Figure \ref{evolution})
    \begin{equation*}
    \Delta_\epsilon(e_0)=\begin{cases}
    \Delta_0(e_0)&\text{with probability }1-3\epsilon+O(\epsilon^2),\\
    (1-2\mu)\,\Delta_0(e_0)&\text{w.p.\ }\epsilon+O(\epsilon^2),\\
    \Delta_0(e_0)+\mu\cdot\Delta_0(e_{-1})\pmod {\mathcal{S}}&\text{w.p.\ }\epsilon+O(\epsilon^2),\\
    \Delta_0(e_0)+\mu\cdot\Delta_0(e_1)\pmod {\mathcal{S}}&\text{w.p.\ }\epsilon+O(\epsilon^2)\quad\text{and}\\
    Z&\text{w.p.\ }O(\epsilon^2),
    \end{cases}
    \end{equation*}
    where $Z$ is some $(-1,1]$-valued random variable. 
    As the distributional invariance $\mathcal{L}\big(\bm{\eta}_0\big)=\nu =\mathcal{L}\big(\bm{\eta}_t\big)$ 
    directly implies $\E\big|\Delta_0(e_0)\big|=\E\big|\Delta_t(e_0)\big|$ 
    for all $t>0$, we can conclude (looking at time $t=\epsilon$ and using the independence of $\bm{\eta}_0$ from 
    the Poisson processes) that
    \[2\,(1+\mu)\,\E\big|\Delta_0(e_0)\big|\leq\E\big|\Delta_0(e_0)+\mu\,\Delta_0(e_{-1})\big|+\E\big|\Delta_0(e_0)+\mu\,\Delta_0(e_1)\big|
    +O(\epsilon),\]
    where the inequality comes from the potentially involved modulo calculation.
    Letting $\epsilon$ tend to $0$ and using the triangle inequality we get
    \begin{equation}\label{discrep}
    2\,\E\big|\Delta_0(e_0)\big|\leq\E\big|\Delta_0(e_{-1})\big|+\E\big|\Delta_0(e_1)\big|.
    \end{equation}
    Since these conclusions similarly apply to any other edge, in place of $e_0$, the convexity condition \eqref{discrep} must
    be an equality, as $\E\big|\Delta_0(e_v)\big|\in[0,1]$, for all $v\in\Z$, would be violated otherwise. Consequently, for invariant $\nu$ it holds
    \begin{equation}\label{discrep2}
    2\,\E\big|\Delta_0(e_0)\big|=\E\big|\Delta_0(e_{-1})\big|+\E\big|\Delta_0(e_1)\big|.
    \end{equation} 
    This in turn implies that $v\mapsto \E\big|\Delta_0(e_v)\big|$ is a linear function and, due to boundedness, it must be constant.
    To get equality in \eqref{discrep} requires
    \begin{align*}
    \E\big|\Delta_0(e_0)+\mu\,\Delta_0(e_{-1})\big|&=\E\big|\Delta_0(e_0)\big|+\mu\,\E\big|\Delta_0(e_{-1})\big|\quad\text{and}\\
    \E\big|\Delta_0(e_0)+\mu\,\Delta_0(e_1)\big|&=\E\big|\Delta_0(e_0)\big|+\mu\,\E\big|\Delta_0(e_1)\big|,
    \end{align*}
    which proves the claim for $v\in\{-1,0\}$ and $t=0$. To arrive at the full claim, observe once more that in the above argument we can
    simply replace $e_0$ by any other edge $e_v$ and that for fixed $v\in \Z$, due to the invariance of $\nu$, \eqref{samesigneq} either
    holds for all times $t\geq0$ or none.
\end{proof}

\vspace*{1em}
\begin{nproof}{of Theorem \ref{thm-invariant}}
	We proceed indirectly by assuming that there is an invariant law $\mathcal{L}\big(\bm{\eta}_0\big)$ that is non-constant, which means that there is no
	$\mathcal{S}$-valued random variable $L$, such that $\eta_0(v)=L$ a.s.\ for all $v\in\Z$. We shall prove that, under this assumption, with positive probability
		\begin{equation}\label{toshow}
	\Delta_t(e_v)\cdot\Delta_t(e_{v+1})<0\quad\text{for some }(v,t)\in\Z\times[0,\infty),
	\end{equation}
	thereby constructing a contradiction to Lemma \ref{samesign}. 

	Given an invariant distribution for $\bm{\eta}_0$, from \eqref{discrep2} we learn that there exists a constant $c\in[0,1]$, such that $\E\big|\Delta_0(e_v)\big|=c$ for all $v\in\Z$,
	and the assumption that $\bm{\eta}_0$ is not almost surely constant forces $c>0$. Using $0\leq\big|\Delta_0(e_v)\big|\leq1$, we can further
	conclude that
	\begin{equation}\label{pos}\Prob\big(\big|\Delta_0(e_v)\big|\geq \tfrac{c}{2}\big)\geq \tfrac{c}{2}\quad\text{for all }v\in\Z.\end{equation}
	Let $K:=\big\lceil\tfrac2c\big\rceil\,\big(\big\lceil\tfrac4c\big\rceil+3\big)+1$ and consider the section $P_K=(V_K,E_K)\subseteq\Z$,
	where $V_K=\{1,\dots,K\}$ and $E_K=\{e_1,\dots, e_{K-1}\}$. Let 
	\[X:=\Big|\big\{e\in E_K;\; \big|\Delta_0(e)\big|\geq\tfrac{c}{2}\big\}\Big|\] be the number of edges in $E_K$ on which initially there is
	an edge difference of at least $\tfrac{c}{2}$. From \eqref{pos}, $\E X\geq\tfrac{c}{2}\,(K-1)\geq\big\lceil\tfrac4c\big\rceil+3$ follows.
	Consequently, $\Prob(A)\geq\Prob(X\geq \E X)>0$, where 
	\[A:=\Big\{X\geq \big\lceil\tfrac4c\big\rceil+3\Big\}.\]
	Conditioned on $A$, we distinguish two cases:
	
	\underline{Case 1:} With positive probability, there are (not necessarily neighboring) edges $e,e'\in E_K$, whose initial edge differences have opposite
	sign, i.e.\ \[\Delta_0(e)\cdot\Delta_0(e')<0.\]	
	Let us choose a pair of such edges with minimal distance, say $e=e_v$ and
	$e'=e_{v+k}$, for some $k\in\{1,\dots,K-2\}$. If $k=1$, we directly arrive at \eqref{toshow}. If $k>1$, however, our choice implies \[\Delta_0(e_{v+1})=\ldots=\Delta_0(e_{v+k-1})=0.\] Let $B$ be the event that the chronologically ordered
	sequence of Poisson events on the edge set $\{e_{v-1},e_v,\dots,e_{v+k},e_{v+k+1}\}$ during the time period $0\leq t\leq 1$ is
	given by $(e_{v+k},e_{v+k-1},\dots,e_{v+2})$. Since $B$ is independent of $A$ -- as the Poisson processes are independent of
	the starting configuration -- we get $\Prob(A\cap B)>0$. Given $A\cap B$, we arrive at $\Delta_1(e_v)\cdot\Delta_1(e_{v+1})<0$,
	which concludes the first case.
	
	\underline{Case 2:} A.s.\ either $\Delta_0(e)\geq0$ for all $e\in E_K$ or $\Delta_0(e)\leq0$ for all $e\in E_K$. By symmetry, we can
	w.l.o.g.\ assume the former. Given $A$, we are guaranteed at least $\big\lceil\tfrac4c\big\rceil+3$ edges $e\in E_K$
	with $\Delta_0(e)\geq\tfrac{c}{2}$. Let $e_v$ be the leftmost, $e_{v+k}$ the rightmost of these edges and observe that
	\[\sum_{j=1}^{k-1}\Delta_0(e_{v+j})\geq\tfrac{c}{2}\,\Big(\big\lceil\tfrac4c\big\rceil+1\Big)\geq2+\tfrac{c}{2}.\]
	Imagine having no updates on neither $e_{v+1}$ nor $e_{v+k-1}$ and plenty on all edges in between these two until time $t=1$,
	such that $\sum_{j=2}^{k-2}\big|\Delta_1(e_{v+j})\big|\leq \tfrac{c}{3}$, cf.\ Corollary \ref{cor}.
	Then the following scenario must occur: An update on $e_{v+j}$ for some $j\in\{2,\dots,k-2\}$ at time
    $t\in[0,1]$ makes at least one of the neighboring edge differences, i.e.\ $\Delta_t(e_{v+j-1})$ or $\Delta_t(e_{v+j+1})$, flip to
    negative sign. This comes from the fact that in the process of balancing the section $\{e_{v+2},\dots,e_{v+k-2}\}$ until time $t=1$
    (if no such sign flip occurs earlier) the amount of positive difference cumulated at either $e_{v+1}$ or $e_{v+k-1}$ must
    exceed $1$ and hence cause a sign flip there.
    
    If $\mu<\tfrac12$, the edge on which the Poisson event occurred that caused the first sign flip retains strictly positive edge difference
    and we arrive at \eqref{toshow}. If $\mu=\tfrac12$ we can argue as in the second part of case 1.

    In conclusion, we showed that if there was an invariant distribution $\nu$ on $\Z^\mathcal{S}$, attributing positive probability
    to non-constant profiles, \eqref{toshow} has to occur with positive probability, which in turn contradicts the invariance in view of Lemma
    \ref{samesign}. This concludes the proof.	
\end{nproof}

\section{Further research} \label{future}
We have established that the compass model with i.i.d.\ uniform initial configuration exhibits weak consensus
in mean (and thus also in probability) and no strong consensus (in any sense). It remains open whether or not there
is weak consensus also in the almost sure sense. While it is known that there is almost sure weak consensus in the standard
Deffuant model with trivial confidence bound on $\Z^n$ (i.e.\ i.i.d.\ $\mathrm{unif}([0,1])$ initial opinions and $\theta=1$, 
cf.\ Thm.\ 3.1 in \cite{Deffuant}), it is still unknown, if in this setting there is any form of strong consensus for
$n\geq2$. With this in mind the question can be put more broadly, namely whether there is \emph{any} (meaningful) model
with almost sure weak consensus but no strong consensus. 

Further, one could ask to what extent our results carry over to natural extensions of the model analyzed here, for instance
the (uniform) compass model in higher dimensions (i.e.\ on $\Z^n, n\geq2$), the compass model with a non-trivial
``confidence bound'' $\theta$ (as in the Deffuant model), or different initial conditions.
Apart from Proposition \ref{notstrong}, which holds both in higher dimensions and with bounded confidence as
remarked earlier, most of our proofs are not robust to such fundamental changes of the model.	
Even though some of our techniques carry over to higher dimensions, we crucially exploit the fact that the maximal degree
in the graph $\Z$ is 2 in the proof of Corollary \ref{cor}.

Concerning the initial configuration, it seems that many of our arguments could be extended to more general distributions.
However, in view of Theorem \ref{thm-invariant} it is clear that some condition must be assumed.
Further, it should be mentioned that monotonicity plays an important role in the analysis of the Deffuant model. In the compass model, where
there is no apparent order in the state space, monotonicity becomes a subtle issue. 
There still is monotonicity in the form of Lemma \ref{monotone}; the edge difference process on a fixed edge itself, however,
is not a supermartingale.

\subsection*{Acknowledgements}
TH is grateful to Mia Deijfen for productive discussions. 
We thank Roland Bauerschmidt for correspondence concerning the dynamic XY model. We are grateful to an anonymous referee for helpful comments.


\vspace{0.5cm}
\makebox[0.8\textwidth][l]{
	\begin{minipage}[t]{\textwidth}
	{\sc \small Nina Gantert\\
	Fakult\"at f\"ur Mathematik,\\
	Technische Universit\"at M\"unchen,\\
	Boltzmannstr. 3,\\
	85748 M\"unchen, Germany.}\\
   gantert@ma.tum.de
	\end{minipage}}

\vspace{0.5cm}
\makebox[0.8\textwidth][l]{
	\begin{minipage}[t]{\textwidth}
	{\sc \small Markus Heydenreich\\
   Mathematisches Institut,\\
   Ludwig-Maximilians-Universit\"at,\\
   Theresienstr. 39,\\
   80333 M\"unchen, Germany.}\\
   m.heydenreich@lmu.de
	\end{minipage}}

\vspace{0.5cm}
\makebox[0.8\textwidth][l]{
	\begin{minipage}[t]{\textwidth}
	{\sc \small Timo Hirscher\\
   Matematiska institutionen,\\
   Stockholms Universitet,\\
   106 91 Stockholm, Sweden.}\\
   timo@math.su.se
	\end{minipage}}

\end{document}